\documentclass[10pt]{article}
\usepackage{amsfonts,color}
 \usepackage{amsmath,amssymb,mathtools}
 \usepackage{epsfig}
\usepackage{lscape}
\usepackage{amssymb}
\usepackage{graphicx}
\usepackage[utf8]{inputenc}	
\usepackage{amsthm,amssymb}
\usepackage{amsfonts}
\usepackage[english]{babel}
\usepackage{dsfont}
\usepackage{bbm}
\usepackage[T1]{fontenc}
\usepackage{colonequals}
\usepackage{geometry}
\usepackage{wrapfig}

\newcommand{\id}{{\boldsymbol{\mathbbm{1}}}}

\allowdisplaybreaks[1]

\makeindex

 \newtheorem{theorem}{Theorem}[section]

 \newtheorem{proposition}[theorem]{Proposition}

 \newcommand{\R}{\mathbb{R}}
 
 \newcommand{\lam}{\lambda}
 \newcommand{\norm}[2][]{\|#2\|_{#1}}

 \DeclareMathOperator{\diag}{diag}
 \DeclareMathOperator{\Sym}{Sym}

 \DeclareMathOperator{\dev}{dev}

\newcommand{\GL}{{\rm GL}}
\newcommand{\SO}{{\rm SO}}

\def\barr{\begin{array}}

\DeclareMathOperator{\tr}{tr}

\def\dd{\displaystyle}

\def\barr{\begin{array}}
\def\earr{\end{array}}
\def\bec#1{\begin{equation}\label{#1}}
\def\becn{\begin{equation*}}
\def\endec{\end{equation}}
\def\endecn{\end{equation*}}

\makeatletter
\let\@fnsymbol\@arabic
\makeatother

\begin{document}
\newgeometry{top=2em,bottom=6.65em}
\title{An ellipticity domain for the distortional Hencky-logarithmic strain energy}
\author{Ionel-Dumitrel Ghiba\thanks{Corresponding author: Ionel-Dumitrel Ghiba, \ \  Lehrstuhl f\"{u}r Nichtlineare Analysis und Modellierung, Fakult\"{a}t f\"{u}r Mathematik,
Universit\"{a}t Duisburg-Essen, Thea-Leymann Str. 9, 45127 Essen, Germany;  Alexandru Ioan Cuza University of Ia\c si, Department of Mathematics,  Blvd.
Carol I, no. 11, 700506 Ia\c si,
Romania; and  Octav Mayer Institute of Mathematics of the
Romanian Academy, Ia\c si Branch,  700505 Ia\c si, email: dumitrel.ghiba@uni-due.de, dumitrel.ghiba@uaic.ro} \quad
and \quad
Patrizio Neff\thanks{Patrizio Neff,  \ \ Head of Lehrstuhl f\"{u}r Nichtlineare Analysis und Modellierung, Fakult\"{a}t f\"{u}r
Mathematik, Universit\"{a}t Duisburg-Essen,  Thea-Leymann Str. 9, 45127 Essen, Germany, email: patrizio.neff@uni-due.de}
\quad and \quad Robert J.\ Martin\thanks{Robert J.\ Martin, \ \ Lehrstuhl f\"{u}r Nichtlineare Analysis und Modellierung, Fakult\"{a}t f\"{u}r Mathematik,
Universit\"{a}t Duisburg-Essen, Thea-Leymann Str. 9, 45127 Essen, Germany; email:   robert.martin@uni-due.de}
}

\maketitle

\begin{abstract}

We describe ellipticity domains for  the isochoric elastic energy
\[
	F\mapsto \|\dev_n\log U\|^2=\bigg\|\log \frac{\sqrt{F^TF}}{(\det F)^{1/n}}\bigg\|^2
=\frac{1}{4}\,\bigg\|\log \frac{C}{(\det C)^{1/n}}\bigg\|^2
\]
for $n=2,3$, where $C=F^TF$ for $F\in {\rm GL}^+(n)$. Here, $\dev_n\log {U} =\log {U}-\frac{1}{n}\,
\tr(\log {U})\cdot\id$
 is the deviatoric part of   the logarithmic strain tensor $\log U$. For $n=2$ we identify the maximal  ellipticity domain, while for $n=3$ we show that the energy is Legendre-Hadamard elliptic in the set
 \[
 	\mathcal{E}_3\bigg(W_{_{\rm H}}^{\rm iso}, {\rm LH}, U, \frac{2}{3}\bigg)\colonequals\bigg\{U\in{\rm PSym}(3) \;\Big|\; \|\dev_3\log U\|^2\leq \frac{2}{3}\bigg\},
 \]
 which is similar to  the von-Mises-Huber-Hencky maximum distortion strain energy criterion.
 Our results complement the characterization of ellipticity domains for the quadratic Hencky energy $
	W_{_{\rm H}}(F)=\mu \,\|\dev_3\log U\|^2+ \frac{\kappa}{2}\,[\tr (\log U)]^2
$, $U=\sqrt{F^TF}$  with  $\mu>0$ and $\kappa>\frac{2}{3}\, \mu$, previously obtained by Bruhns et al.
\\
\vspace*{0.15cm}
\\
\textbf{Mathematics Subject Classification}: 74B20, 74G65, 26B25
\\
\vspace*{0.15cm}
\\
\textbf{Key words}: rank-one convexity, nonlinear elasticity, Hencky energy, logarithmic strain, ellipticity domain, isochoric energy, volumetric-isochoric split, constitutive inequalities, Legendre-Hadamard condition.
\end{abstract}
\tableofcontents

\restoregeometry
\newgeometry{top=6em,bottom=6em,left=7em,right=7em}
\newpage

\section{Introduction}
The quadratic Hencky energy
\[
	W_{_{\rm H}}(F) = \mu \,\|\log U\|^2+ \frac{\lambda}{2}\,[\tr (\log U)]^2 
	= \mu \,\|\dev_n \log U\|^2 + \frac{\kappa}{2}\,[\log(\det U)]^2\,,
\]
where $\kappa\geq0$ denotes the bulk modulus and $\mu,\lambda$ are the Lam\'{e} constants with $\mu>0$ and $3\,\lambda+2\,\mu\geq0$, has recently been shown to have a fundamental geometric property which uniquely characterizes it among all hyperelastic formulations: it measures the (squared) geodesic distance of the deformation gradient $F=\nabla\varphi\in\GL^+(n)\colonequals\{X\in \mathbb{R}^{n\times n}\,|\, \det X>0\}$ to the special orthogonal group $\SO(n)$ \cite{Neff_Osterbrink_Martin_hencky13,NeffEidelOsterbrinkMartin_Riemannianapproach}.
Here, $U=\sqrt{F^T F}$ is the right stretch tensor and $\dev_n\log {U} =\log {U}-\frac{1}{n}\,\tr(\log {U})\cdot\id$ is the deviatoric part of the Hencky strain tensor $\log U$, $\id$ denotes the identity tensor on $\R^{n\times n}$, $\|{X}\|^2=\langle {X},{X}\rangle$ is the Frobenius tensor norm and $\tr{(X)}=\langle {X},{\id}\rangle$ is the trace of $X\in\R^{n\times n}$.
The Hencky energy $W_{_{\rm H}}$ was originally introduced by Heinrich Hencky in 1928 \cite{Hencky29a,henckyTranslation}, see also  Richter's 1948 paper \cite[eq. 7.3]{richter1948isotrope}. However, in an 1931 article in the Journal of Rheology \cite{Hencky31}, Hencky also considered elastic energy potentials of the form
\begin{equation}
\label{eq:hencky1931}
	W_{1931}(F) = \mu\,\norm{\dev_3\,\log U}^2 + h(\det U)\,.
\end{equation}
Here, the volumetric part $h:\R_+\colonequals(0,\infty)\to\R$ of the energy is a function to be determined by experiments. In an 1933 article \cite{hencky1933elastic}, he suggested an even more general expression for describing the elastic behaviour of vulcanized rubber:
\[
	W_{1933}(F) =\mu\,\sum_{i=1}^n f\left(\log\frac{\lambda_i}{(\lambda_1\,\lambda_2\,\lambda_3)^{1\!/\!3}}\right) \;+\; h(\det U) = \mu\,\widetilde{f}(\dev_3\log U) + h(\det U)\,,
\]
where $\widetilde{f}:\Sym(3)\to\R$ is an isotropic function in Valanis-Landel form \cite{valanis1967strain,NeffMartin14}. The ellipticity of $W_{1931}$, provided that $h$ is convex on $\R_+$, depends only on the ellipticity properties of the isochoric term $\norm{\dev_3\,\log U}^2$ measuring the purely distortional part of the deformation, which we investigate in this article.

The necessity of finding an ellipticity domain for the isotropic invariant $\|\dev_n \log U\|^2$  of the logarithmic strain tensor $\log U$ (see \cite{neff2013hencky,neff2013hencky,Neff_Osterbrink_Martin_hencky13,NeffEidelOsterbrinkMartin_Riemannianapproach,Walton05}) arises from the observation that the isochoric part
\[
	W_{_{\rm H}}^{\text{\rm iso}}\left(F\right)\colonequals\mu \,\|{\rm dev}_n\log U\|^2 
\]
of the quadratic Hencky energy is not rank-one convex even in ${\rm SL}(n)\colonequals\{X\in {\rm GL}^+(n)\;|\det{X}=1\}$ for $n=2,3$ (see \cite{NeffGhibaLankeit}).
The understanding of loss of ellipticity is of fundamental importance in nonlinear elasticity \cite{Ogden83,Silhavy97,Eremeyev4,cism_book_schroeder_neff09,ndanou2014criterion,ndanou2015piston,favrie2014thermodynamically,Hutchinson82}.

It is easy to show that a given hyperelastic formulation is not rank-one convex.
In general, it is also clear that there exists a neighborhood of the identity tensor $\id$ where the formulation is Legendre-Hadamard elliptic (LH-elliptic). What is difficult, however, is to precisely describe the maximal domain of ellipticity: although we are able to numerically determine the maximal ellipticity domain for $\|\dev_n\log U\|^2$, in this paper we choose an intermediary way in that we analytically describe a large set in which the energy is LH--elliptic. The numerically obtained visualization in Fig.\ \ref{f12} indicates that this subset, expressed in terms of certain transformations of the principal stretches, is in fact the ellipse inscribed in the maximal ellipticity domain. For practical applications (like the coupling with elastoplasticity), knowing such a domain is mostly sufficient.

The analysis in this paper  is also motivated by the results established for $n=3$ by Bruhns et al. \cite{Bruhns01,Bruhns02JE} (see also
\cite{knowles1975ellipticity,glugegraphical} in order to compare the domains of ellipticity  obtained in nonlinear elastostatics
for a special material), who found an ellipticity domain for the quadratic Hencky strain energy $W_{_{\rm H}}$: they showed that $W_{_{\rm H}}$ satisfies the
Legendre-Hadamard condition for all principal stretches $\lambda_i$ with
$
\lambda_i\in[0.21162...,\sqrt[3]{e}]=[0.21162...,1.39561...]
$,
provided that the additional condition $\lambda>0$ holds. This result, however, is not applicable to the deviatoric quadratic Hencky energy $\|\dev_n \log U\|^2$, which corresponds to the case $\lambda = -\frac{2\,\mu}{3} < 0$.

It might also be worthwhile to find a scalar function of the isotropic invariant $\|\dev_n \log U\|^2$ such that the composition  is elliptic over  ${\rm GL}^+(n)$. Indeed, in the two-dimensional case we have identified such functions  \cite{NeffGhibaLankeit,ghiba2015exponentiated,GhibaNeffMartin2015,NeffGhibaPoly}, namely the so-called \emph{exponentiated Hencky energies}
\begin{align}\label{thdefHen}\hspace{-2mm}
 W_{_{\rm eH}}^{\text{\rm iso}}\left(F\right)\colonequals\dd\frac{\mu}{k}\,e^{k\,\|{\rm dev}_2\log U\|^2}, \qquad F\in {\rm GL}^+(2)\,,
\end{align}
where  $k$ is an additional dimensionless
parameter.   In \cite{ghiba2015exponentiated} it is shown that these energies are  polyconvex for $k\geq \frac{1}{4}$. In fact, $W_{_{\rm eH}}^{\text{\rm iso}}\left(F\right)$ is polyconvex if and only if $k\geq \frac{1}{4}$, see \cite{GhibaNeffMartin2015}, while $\norm{\dev_2 \log U}^2$ is not overall rank-one convex. For $n=3$, however, such a function is not yet known.

Knowles and Sternberg \cite{knowles1975ellipticity} have established a criterion for rank-one convexity (ellipticity) which is necessary and sufficient for $n=2$ but only necessary for $n=3$ (see also \cite{aubert1995necessary,aubert1985conditions,ball1984differentiability,davies1991simple,knowles1978failure,abeyaratne1980discontinuous} for alternative proofs). For $n=3$, necessary and sufficient conditions for ellipticity were given by Simpson and Spector \cite{Simpson87}, while for arbitrary dimension they were established for the first time by \v{S}ilhav\'{y} \cite{SilhavyPRE99} in terms of the copositivity of certain matrices.
The necessary and sufficient conditions introduced by Dacorogna \cite{Dacorogna01} were obtained by combining a result established by \v{S}ilhav\'y \cite{SilhavyPRE99} with one result on copositive matrices in dimension 3 by Hadeler \cite{hadeler1983copositive}. In this paper, we use Dacorogna's sufficient criterion \cite{Dacorogna01} for arbitrary $n$, which can be applied more easily than Dacorogna's necessary and sufficient conditions \cite[Theorem 5]{Dacorogna01}, and for $n=2$ is equivalent to the necessary and sufficient criterion previously shown by Knowles and Sternberg \cite{knowles1975ellipticity}.

We will, however, not go into further detail on the general importance of rank-one convexity.
For more information on this topic, we refer to the comprehensive books \cite{Dacorogna08,Silhavy97,DacorognaMarcellini} and to the papers \cite{Dacorogna01,SilhavyPRE99,Silhavy01,Silhavy03,Silhavy05}.

\section{Preliminaries}\setcounter{equation}{0}

An energy $W:{\rm GL}^+(n)\rightarrow\mathbb{R}$ is called rank-one convex \cite[page 352]{Ball77} on ${\rm GL}^+(n)$ if it is convex on all closed line segments in ${\rm GL}^+(n)$ with end points differing by a matrix of rank one, i.e.\ if
\begin{equation}
\label{eq:ROdefinition}
	W(F+(1-\theta)\, \xi\otimes \eta)\leq \theta \,W( F)+(1-\theta) W(F+\xi\otimes \eta)
\end{equation}
for all $F\in {\rm GL}^+(n)$, $\theta\in[0,1]$ and all $\,\, \xi,\, \eta\in\mathbb{R}^n$ with $F+t\, \xi\otimes \eta\in {\rm GL}^+(n)$ for all
$t\in[0,1]$, where $\xi\otimes\eta$ denotes the dyadic product. Using definition \eqref{eq:ROdefinition}, in  \cite{Neff_Diss00} it was shown for the first time that the mapping $F\mapsto\norm{\dev_n\log U}^2$ is not rank-one convex on all of $\GL^+(n)$.

Since ${\rm GL}^+(n)$ is an open subset of $\R^{n\times n}$, an energy $W:{\rm GL}^+(n)\rightarrow \mathbb{R}$ of class $C^2$ is rank-one convex if and only if it is Legendre-Hadamard elliptic (LH-elliptic) at all points $F\in {\rm GL}^+(n)$:
\begin{equation}
\label{eq:LHdefinition}
	D^2_F W(F).\,(\xi\otimes\eta,\xi\otimes\eta)\geq0 \quad\text{ for all }\;\xi,\eta\in\mathbb{R}^n\,.
\end{equation}
Note carefully that, by this definition, rank-one convexity is strictly a global concept: a function on $\GL^+(n)$ is either rank-one convex or it isn't. Legendre-Hadamard ellipticity, on the other hand, is also well-defined as a local property: a function $W:\GL^+(n)\to\R$ is called LH--elliptic (or simply \emph{elliptic}) on a set $\mathcal{E}\subset\GL^+(n)$ if \eqref{eq:LHdefinition} holds for all $F\in\mathcal{E}$. In this case, $\mathcal{E}$ is also called a \emph{domain of ellipticity} or \emph{ellipticity domain} for $W$. We also use the term \emph{maximal ellipticity domain} to refer to the set of \emph{all} points in which a function is LH--elliptic.

Let us remark that whereas Bruhns, Xiao and Mayers \cite{Bruhns01,Bruhns02JE} directly used definition \eqref{eq:LHdefinition} for finding an ellipticity domain of the quadratic Hencky energy, in this paper we do not calculate the second derivative $D^2_F W(F)$ of the energy $W(F)=\norm{\dev_n\log U}^2$.
Instead, we will consider the representation of the isotropic energy $W$ in terms of the principal stretches and utilize criteria applicable to this representation.

Next we recall some of these useful results about LH--ellipticity as well as some properties of the deviatoric part  of the strain tensor $\log U$.

\subsection{Criteria for LH--ellipticity based on principal stretches}\label{ROsect}

In the three-dimensional case, our purpose is to identify an ellipticity domain, but not necessarily the maximal one, for the energy $F\mapsto\|\dev_3 \log U\|^2$.
We therefore need a suitable \emph{sufficient} criterion for LH--ellipticity. The following theorem was given by Dacorogna \cite[Proposition 7]{Dacorogna01} in the form of a criterion for rank-one convexity, i.e.\ for ellipticity on all of $\GL^+(n)$. It can easily be seen from his proof that the local form given here holds as well; note that the requirement that the set $\mathcal{E}$ is open or the closure of an open set ensures that every $F\in\mathcal{E}$ can be written as the limit of a sequence $(F_k)_k\subset\mathcal{E}$ of matrices with pairwise different singular values, which is utilized in Dacorogna's proof. The criterion has previously been used by Gl\"uge and Kalisch \cite{glugegraphical} in a similar way.
\begin{theorem}\label{dacorogna5}
 Let  ${W}:{\rm GL}^+(n)\rightarrow\mathbb{R}$ be an objective and isotropic function of class $C^2$ with the representation in terms of the singular values of
 $U$ via $W(F)=g(\lambda_1,\lambda_2,...,\lambda_n)$, where $g\in C^2(\R_+^n,\R)$ is symmetric. Further, let $\mathcal{E}\subset\GL^+(n)$ be an open set or the closure of an open set. Define $\widetilde{\mathcal{E}}\subset\R_+^n$ as follows:
 \[
 	(\lambda_1,\dotsc,\lambda_n)\in\widetilde{\mathcal{E}} \ \text{ if and only if there exists }\ F\in\mathcal{E}\text{ such that }\lambda_1,\dotsc,\lambda_n\text{ are the singular values of }F\,.
 \]
 Then $W$ is Legendre-Hadamard elliptic at all $F\in\mathcal{E}$ if the following four sets of conditions hold:
 \begin{alignat}{3}
 &\text{i) }&&\!\!\underbrace{\frac{\partial^2 g}{\partial \lambda_i^2 }\geq 0}_{\text{\rm ``TE--inequalities''}} \mathrlap{\text{ for every }\; i=1,2,..,n \;\text{ and all }\;(\lambda_1,\lambda_2,...,\lambda_n)\in\widetilde{\mathcal{E}}\,,}\label{DBE1}\\[1em]
 &\text{ii) \ }&&\text{ for all }\;i\neq j\,,\nonumber\\[1em]
 &&&\dd\underbrace{\frac{\lambda_i\frac{\partial g}{\partial \lambda_i}-\lambda_j\frac{\partial g}{\partial \lambda_j}}{\lambda_i-\lambda_j}\geq
 0}_{\text{\rm ``BE--inequalities''}} &&\text{for all }\; (\lambda_1,\lambda_2,...,\lambda_n)\in\widetilde{\mathcal{E}} \;\text{ with }\;   \lambda_i\neq \lambda_j\,, \label{DBE2}\\
 &&&\dd\frac{1}{n-1}\,\sqrt{\frac{\partial^2 g}{\partial \lambda_i^2
 }\frac{\partial^2 g}{\partial \lambda_j^2 }}+\frac{\partial^2 g}{\partial \lambda_i \partial \lambda_j} +\frac{\frac{\partial g}{\partial
                             \lambda_i }- \frac{\partial g}{\partial \lambda_j}}{ \lambda_i -
                             \lambda_j}\geq 0  &&\text{for all }\; (\lambda_1,\lambda_2,...,\lambda_n)\in\widetilde{\mathcal{E}} \;\text{ with }\;   \lambda_i\neq \lambda_j\,, \label{DBE3}\\
 &&&\frac{1}{n-1}\,\sqrt{\frac{\partial^2 g}{\partial \lambda_i^2
 }\frac{\partial^2 g}{\partial \lambda_j^2 }}-\frac{\partial^2 g}{\partial \lambda_i \partial \lambda_j} +\frac{\frac{\partial g}{\partial
                             \lambda_i }+ \frac{\partial g}{\partial \lambda_j}}{ \lambda_i +
                             \lambda_j}\geq 0 &\qquad&\text{for all }\; (\lambda_1,\lambda_2,...,\lambda_n)\in\widetilde{\mathcal{E}}\,. \label{DBE4}
 \end{alignat}
For $n=2$, the conditions are also necessary.\hfill $\blacksquare$
\end{theorem}

 Here $\mathbb{R}_+=(0,\infty)$. The necessary and sufficient conditions of this theorem in the case $n=2$ are the same as established by  Knowles and Sternberg  \cite{knowles1976failure,knowles1978failure}, see also \cite[page 318]{Silhavy97}.

Dacorogna \cite[page 6]{Dacorogna01} also explains that due to the permutation symmetry of $g$, it is enough to establish only $4$ inequalities: one TE--inequality (tension-extension inequality) for $i=1$, one BE--inequality (Baker-Ericksen inequality) for $i=1$ and $j=3$ and two other inequalities from \eqref{DBE3}, \eqref{DBE4} 
for $i=1$, $j=3$. Note carefully that this remark is valid only when one considers the question whether a function is rank-one convex, i.e.\ LH--elliptic on all of $\GL^+(n)$; if a specific domain $\mathcal{E}$ is considered, then the corresponding set $\widetilde{\mathcal{E}}$, which consists of all $(\lambda_1,\dotsc,\lambda_n)\in\R_+^n$ which are singular values of some $F\in\mathcal{E}$, has to be invariant under permutations in order to reduce the number of inequalities. 

If, on the other hand, one wants to completely characterize  the maximal ellipticity domain for an energy in spatial dimension $n=3$, then the necessary and sufficient conditions of Dacorogna \cite[Theorem 5]{Dacorogna01} are better suited. In this set of conditions, one has to show 10 inequalities. However, due to some other symmetries and invariance properties, and since the BE--inequalities are always satisfied by $\|\dev_n\log U\|^2$, there remain 5 inequalities in the necessary and sufficient conditions of Dacorogna which have to be checked in order to study the ellipticity of the energy $\|\dev_3\log U\|^2$. We do not use this criterion in the analytic part of this article.

\subsection{Auxiliary remarks}
The norm of the deviator in $\R^{n\times n}$ is given by
$
\norm{\dev_n\,\diag(\xi_1,\dotsc,\xi_n)}^2 = \smash{\frac{1}{n}\sum\limits_{i,j=1,i<j}^n} (\xi_i-\xi_j)^2\,.
$
Thus, for  $F\in{\rm GL}^+(n)$ with singular values $\lam_1,\lam_2,...,\lam_n$, it follows that
 \begin{align}\label{wfn}
\|\dev_n\log U\|^2=g(\lam_1,\lam_2,...,\lam_n),
\end{align}
where the function  $g:\mathbb{R}_+^n\rightarrow [0,\infty)$ is given by
 \begin{align}\label{wf3}
g(\lambda_1,\lambda_2,...,\lambda_n)=\frac{1}{n}\sum\limits_{i,j=1,i<j}^n (\log{\lambda_i}-\log{\lambda_j})^2=\frac{1}{n}\sum\limits_{i,j=1,i<j}^n \log^2\frac{\lambda_i}{\lambda_j}.
\end{align}

Note that the function $g$ is invariant under scaling:
\begin{align}
  g(a\,\lambda_1,a\,\lambda_2,...,a\,\lambda_n)=g(\lambda_1,\lambda_2,...,\lambda_n) \qquad \text{for all}\quad a>0.
\end{align}
Hence, for the function $g$ corresponding to our energy $F\mapsto  \|\dev_n\log U\|^2$, the inequalities in Dacorogna's criterion are also
invariant under scaling, see \cite{NeffGhibaLankeit} for further details.
Therefore, for an arbitrary scaling factor $a>0$, the function $g$ satisfies the required inequalities from Dacorogna's criterion (Theorem \ref{dacorogna5}) in a point
$(\widetilde{\lambda}_1,\widetilde{\lambda}_2,...,\widetilde{\lambda}_n)
 =(a\,\lambda_1,a\,\lambda_2,...,a\,\lambda_n)$
if and only if it satisfies them in the point
$(\lambda_1,\lambda_2,...,\lambda_n)$.

Since $\norm{\dev_n\,\log U}^2$ linearizes to $\norm{\dev_n\varepsilon}^2$, where $\varepsilon$ denotes the linearized strain tensor, it is obvious that the maximal ellipticity domain of $\|\dev_n\log U\|^2$	contains a neighborhood of $\id$. Moreover, the above considerations show that this domain is an (unbounded) cone containing $\id$. In the following we will exploit this insight.

\section{The two-dimensional case }\setcounter{equation}{0}
Using the ellipticity conditions by Knowles and Sternberg, i.e.\ Theorem \ref{dacorogna5} for $n=2$, we obtain:
\begin{proposition}\label{crank}
The maximal ellipticity domain of the energy  $F\mapsto \|\dev_2\log U\|^2$, $F\in {\rm GL}^+(2)$  is
\begin{align}\label{condrank}
\mathcal{E}_2\left(W_{_{\rm H}}^{\rm iso}, {\rm LH}, U, \frac{1}{2}\right)\colonequals\left\{U\in{\rm PSym}(2)\, |\, \|\dev_2\log U\|^2\leq \frac{1}{2}\right\}.
\end{align}
\end{proposition}
\begin{proof}
We will prove this result using the necessary and sufficient conditions given by Theorem \ref{dacorogna5} for $n=2$ together with the identity \eqref{wfn}. To this aim, we need to compute
\begin{alignat}{2}
\frac{\partial g}{\partial \lambda_1}&=\frac{1}{\lambda_1}\, \log \frac{\lambda_1}{\lambda_2}\,,& \qquad\qquad \frac{\partial g}{\partial \lambda_2}&=-\frac{1}{\lambda_2}\, \log \frac{\lambda_1}{\lambda_2}\,,\\
\frac{\partial^2 g}{\partial \lambda_1^2}&=\frac{1}{\lambda_1^2}\, \left[-  \log \frac{\lambda_1}{\lambda_2}+1\,\right]\,,&
\frac{\partial^2 g}{\partial \lambda_2^2}&=\frac{1}{\lambda_2^2}\, \left[ \log \frac{\lambda_1}{\lambda_2}+1\,\right]\,,\qquad\qquad
 \frac{\partial^2 g}{\partial \lambda_1\partial \lambda_2}=-\frac{1}{\lambda_1\, \lambda_2}\notag
\end{alignat}
and verify that inequalities \eqref{DBE1}--\eqref{DBE4} hold if and only if $(\lambda_1,\lambda_2)\in\widetilde{\mathcal{E}}_2$, where the set $\widetilde{\mathcal{E}}_2$ of singular values corresponding to the domain $\mathcal{E}=\mathcal{E}_2\left(W_{_{\rm H}}^{\rm iso}, {\rm LH}, U, \frac{1}{2}\right)$ is given by
\[
	\widetilde{\mathcal{E}}_2 = \left\{(\lambda_1,\lambda_2)\in\R_+^2 \;\Big|\; \log^2\frac{\lambda_1}{\lambda_2} \leq 1\right\}\,.
\]
The TE--inequalities  of Theorem \ref{dacorogna5} are equivalent to
\begin{align}\label{ineq1}
&- \log \frac{\lambda_1}{\lambda_2}+1 \geq 0\,,\qquad \log \frac{\lambda_1}{\lambda_2}+1\geq 0\,, 
\end{align}
while the BE--inequalities are satisfied everywhere for convex functions of $\log U$ \cite{NeffGhibaLankeit} and thus in particular by $\|\dev_2\log U\|^2$.
The inequalities \eqref{DBE2} and \eqref{DBE3} are equivalent to
\begin{align}
&\sqrt{ \left[-  \log \frac{\lambda_1}{\lambda_2}+1\right]\,\left[ \log \frac{\lambda_1}{\lambda_2}+1\right]}-1+\frac{\lambda_1+\lambda_2}{\lambda_1-\lambda_2}  \log \frac{\lambda_1}{\lambda_2}\geq 0 \quad\text{ if }\; \lambda_1\neq \lambda_2\label{ineq4a}\,, 
\\
&\sqrt{ \left[-  \log \frac{\lambda_1}{\lambda_2}+1\right]\,\left[ \log \frac{\lambda_1}{\lambda_2}+1\right]}+1-\frac{\lambda_1-\lambda_2}{\lambda_1+\lambda_2}  \log \frac{\lambda_1}{\lambda_2}\geq 0\,. \label{ineq4b}
\end{align}
 Since all these inequalities are symmetric in $\lambda_1$ and $\lambda_2$ (and thus the ellipticity domain is invariant w.r.t.\ the transformations $\lambda_1\mapsto \lambda_2$, $\lambda_2\mapsto \lambda_1$) we may assume that
$\lambda_1\geq \lambda_2$, i.e.\ that $t\colonequals\frac{\lambda_1}{\lambda_2}\geq 1$.
Geometrically speaking, considering this substitution means that it is necessary and sufficient to prove that the inequalities \eqref{ineq1}, \eqref{ineq4a} and \eqref{ineq4b} are satisfied along all lines $\lambda_1=t\, \lambda_2$, $t\geq 1$.

Thus, the inequalities \eqref{ineq1}, \eqref{ineq4a} and \eqref{ineq4b} are satisfied if and only if the following inequalities hold:
\begin{align}&-  \log t+1 \geq 0,\qquad \qquad \log t+1\geq 0\,,\label{ineq6a}\\ 
&\sqrt{ \left[-  \log t+1\right]\,\left[ \log t+1\right]}-(1-\log t)+\frac{2}{t-1}\, \log t\geq 0\,,\label{ineq6b}\\ 
 &\sqrt{ \left[-  \log t+1\right]\,\left[ \log t+1\right]}+(1+\log t)+\frac{2}{t+1} \, \log t\geq 0\,.\label{ineq6c} 
 \end{align}

 Since $t\geq 1$ and  $\sqrt{ \left[-  \log t+1\right]\,\left[ \log t+1\right]}\geq (1-\log t)$ for all $t>1$, we find that the inequalities \eqref{ineq6b} and \eqref{ineq6c} are redundant in  the set of inequalities describing the domain of ellipticity.  In conclusion, the independent inequalities describing the ellipticity domain are
\begin{align}\label{ineq9}
	-\log t+1 \geq 0,\qquad \qquad \log t+1\geq 0\,,
\end{align}
which can equivalently be expressed as
$
1 \leq \log^2 t = \log^2\frac{\lambda_1}{\lambda_2}
$.
Therefore we deduce that the ellipticity conditions are satisfied if and only if $(\lambda_1,\lambda_2)\in\widetilde{\mathcal{E}}_2$, i.e. if and only if $\|\dev_2\log U\|^2=\frac{1}{2}\log^2\frac{\lambda_1}{\lambda_2}\leq 1$, and the proof is complete.
\end{proof}
\section{The three-dimensional case }\setcounter{equation}{0}
For $n=3$, we consider the substitution
\begin{align}\label{lab}
\frac{\lambda _1}{\lambda _2}=e^a,\qquad \frac{\lambda _2}{\lambda _3}=e^b,\qquad \frac{\lambda _3}{\lambda _1}=e^{-(a+b)}.
\end{align}
Then
\begin{align}
\frac{\lambda_1}{e^a}=\frac{\lambda_2}{1}=\frac{\lambda_3}{e^{-b}}\equalscolon t\,,
\end{align}
which means that $(\lambda_1,\lambda_2,\lambda_3)$ belongs to the line which passes through $(0,0,0)$ and an arbitrary point $(e^a,1,e^{-b})$ in the plane $\lambda_2=1$. According to the invariance properties of the energy and of the conditions for ellipticity given in our preliminaries, it is enough to study the resulting inequalities only in the plane $\lambda_2=1$. 

Numerical calculations indicate that the three-dimensional maximal domain of ellipticity is that for which $(a,b)$ in \eqref{lab} belongs to the two-dimensional domain described in Fig.\ \ref{f11}. However, since it is difficult to characterize the maximal ellipticity domain,  we consider a significant  large subdomain of it (see Fig.\ \ref{f12}).

\begin{figure}[h!]
\hspace*{1cm}
\begin{minipage}[h]{0.4\linewidth}
\vspace*{3.36em}
\centering
\includegraphics[scale=0.6]{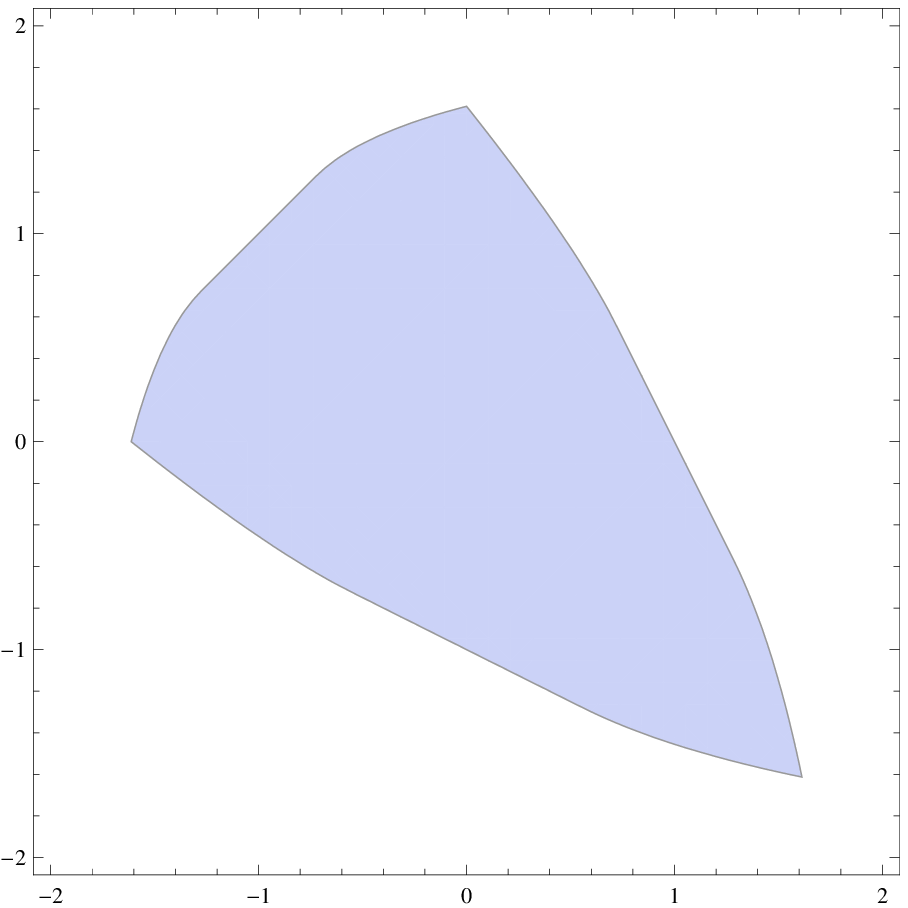}
\centering
\caption{The maximal ellipticity domain in terms of $(a,b)$ obtained numerically  after implementation of the necessary and sufficient criterion of Dacorogna \cite[Theorem 5]{Dacorogna01}, which we do not use in the analytic part of this article.}
\label{f11}
\end{minipage}
\qquad
\begin{minipage}[h]{0.4\linewidth}
\centering
\includegraphics[scale=0.6]{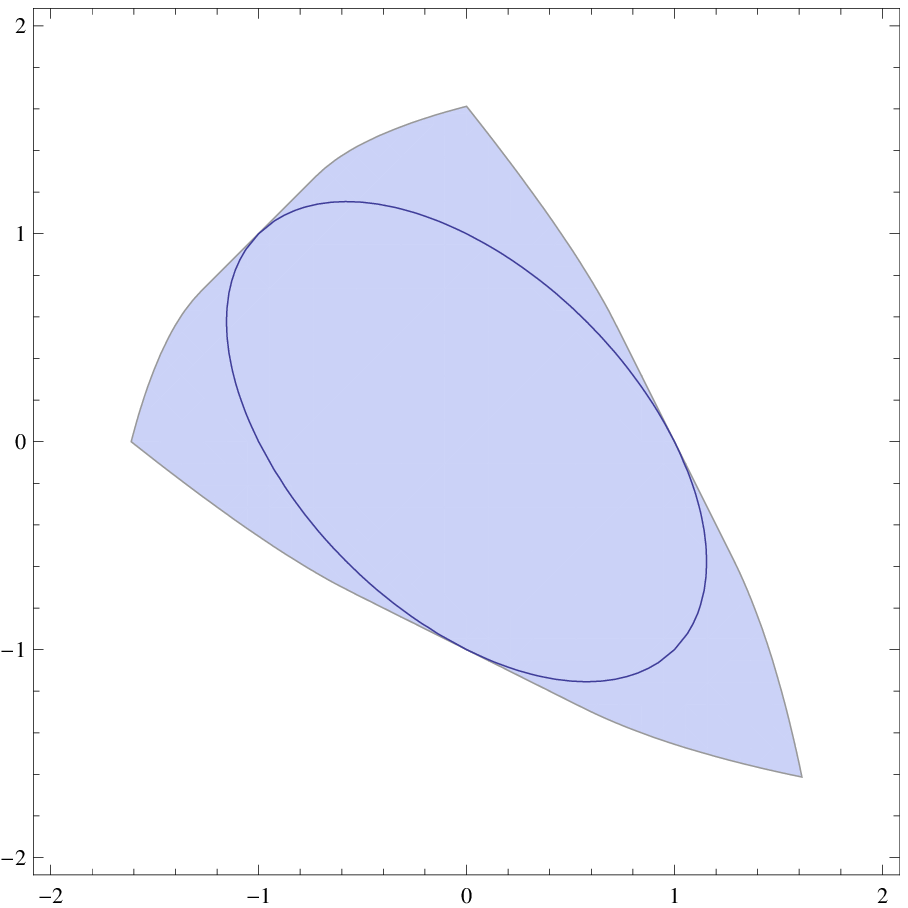}
\centering
\caption{We consider the full ellipse which is contained in the maximal ellipticity domain and which touches its boundary. }
\label{f12}
\end{minipage}
\end{figure}%

\begin{proposition}\label{crank3}
The maximal ellipticity domain of the energy  $F\mapsto \|\dev_3\log U\|^2$, $F\in {\rm GL}^+(3)$  contains the ellipticity domain
\begin{align}\label{condrank3}
\mathcal{E}_3\left(W_{_{\rm H}}^{\rm iso}, {\rm LH}, U, \frac{2}{3}\right)\colonequals\left\{U\in{\rm PSym}(3) \;\Big|\; \|\dev_3\log U\|^2\leq \frac{2}{3}\right\}\,.
\end{align}
\end{proposition}
\begin{proof}
According to \ref{dacorogna5}, we have to show that inequalities \eqref{DBE1}--\eqref{DBE4} hold for all $(\lambda_1,\lambda_2,\lambda_3)\in\widetilde{\mathcal{E}}_3$, where the set $\widetilde{\mathcal{E}}_3$ of singular values corresponding to the domain $\mathcal{E}_3\left(W_{_{\rm H}}^{\rm iso}, {\rm LH}, U, \frac{2}{3}\right)$ is given by
\[
	\widetilde{\mathcal{E}}_3 = \left\{(\lambda_1,\lambda_2,\lambda_3)\in\R_+^3 \;\Big|\; \log^2\frac{\lambda_1}{\lambda_2} + \log^2\frac{\lambda_1}{\lambda_3} + \log^2\frac{\lambda_2}{\lambda_3} \;\leq\; 2\right\}\,.
\]
Let us first observe that, under the substitution \eqref{lab},
\[
	(\lambda_1,\lambda_2,\lambda_3)\in\widetilde{\mathcal{E}}_3 \qquad\Leftrightarrow \qquad a^2+b^2+a\, b\leq 1\,,
\]
which means that $(a,b)$ belongs to the full ellipse  $a^2+b^2+a\, b=1$. This domain is invariant under the three transformations
\begin{align}\label{substitution}
a&=x+y, \quad b=-y; \qquad\qquad\qquad\quad
a=y, \quad b=x; \qquad\qquad\qquad\quad
a=-y, \quad b=-x
\end{align}
in the sense that if $(x,y)$ belongs to the full ellipse  $w^2+z^2+w\, z=1$, then $(a,b)$ belongs also
to this domain and vice versa.
This invariance property is equivalent to the (obvious) invariance of the set $\widetilde{\mathcal{E}}_3$ under permutations of $\lambda_1$, $\lambda_2$ and $\lambda_3$. From this symmetry, it follows that we only need to consider the inequalities in Dacorogna's criterion for $i=1$ and $j=3$.
We compute
\begin{align}
\frac{\partial g}{\partial \lambda_1}&=\frac{2}{3\,{\lambda _1}}\,{\left(\log \frac{\lambda _1}{\lambda _2}\,-\log \frac{\lambda _3}{\lambda _1}\,\right)}, \qquad\ \ \
\frac{\partial g}{\partial \lambda_3}=-\frac{2}{3\,{\lambda _3}}\,\left(\log \frac{\lambda _2}{\lambda _3}\,-\log \frac{\lambda _3}{\lambda _1} \right),\\\
\frac{\partial^2 g}{\partial \lambda_1^2}&=\frac{2}{3\,{\lambda _1^2}}\,\left(\log \frac{\lambda _3}{\lambda _1}\,- \log \frac{\lambda _1}{\lambda _2}\,+2\right),\quad
\frac{\partial^2 g}{\partial \lambda_3^2}=\frac{2}{3\,{\lambda _3^2}}\,\left(  -\log \frac{\lambda _3}{\lambda _1}\,+\log \frac{\lambda _2}{\lambda _3}\,+2\right),\quad
\frac{\partial^2 g}{\partial \lambda_3\partial \lambda_1}=-\frac{2}{3\,\lambda _1 \lambda _3}\,.\notag
\end{align}
{The TE--inequality} for $i=1$ is equivalent to
\begin{align}
 2+a-b\geq 0,
\end{align}
while {the BE--inequality} for $i=1, j=3$ reads
\begin{align}
\frac{\lambda_3\frac{\partial g}{\partial \lambda_3}-\lambda_1\frac{\partial g}{\partial \lambda_1}}{\lambda_3-\lambda_1}=\frac{1}{3\,\lambda_1}\frac{-2 \left(a+2\,b\right)-2 \left(2\,a+b\right)}{e^{-(a+b)}-1}
=\frac{1}{\lambda_1}\frac{-2 \left(a+b\right)}{e^{-(a+b)}-1}\geq 0\notag
\end{align}
and is always satisfied.
Moreover, we compute
  \begin{align}
 \frac{1}{2}\,\sqrt{\frac{\partial^2 g}{\partial \lambda_3^2}\frac{\partial^2 g}{\partial \lambda_1^2}}&-\frac{\partial^2 g}{\partial \lambda_3\partial
 \lambda_1}+\frac{\frac{\partial g}{\partial \lambda_3}+\frac{\partial g}{\partial \lambda_1}}{\lambda_3+\lambda_1}\notag\\
 &=\frac{2}{3\,\lambda_3\,\lambda_1}\Big[\frac{1}{2}\,\sqrt{ \left(2+a+2\,b\right)\left( 2-2\,a-b\right)}+1+ \frac{-\lambda _1\, \left(a+2\,b\right) +\lambda_3\,\left(2\,a+b\right)}{\lambda_3+\lambda_1} \Big]\notag\\
 &=\frac{2}{3\,\lambda_3\,\lambda_1}\Big[\frac{1}{2}\,\sqrt{\left(2+a+2\,b\right) \left( 2-2\,a-b\right) }+1 +\frac{- e^{a+b}\left(a+2\,b\right) +\left(2\,a+b\right)}{1+e^{a+b}}\Big]\notag
 \end{align}
and
  \begin{align}
 \frac{1}{2}\,\sqrt{\frac{\partial^2 g}{\partial \lambda_3^2}\frac{\partial^2 g}{\partial \lambda_1^2}}&+\frac{\partial^2 g}{\partial \lambda_3\partial
 \lambda_1}+\frac{\frac{\partial g}{\partial \lambda_3}-\frac{\partial g}{\partial \lambda_1}}{\lambda_3-\lambda_1}\notag\\
 &=\frac{2}{3\,\lambda_3\,\lambda_1}\Big[\frac{1}{2}\,\sqrt{ \left(2+a+2\,b\right) \left( 2-2\,a-b\right)}-1+\frac{-\lambda _1\, \left(a+2\,b\right) -\lambda_3\,\left(2\,a+b\right)}{\lambda_3-\lambda_1} \Big]\notag\\
 &=\frac{2}{3\,\lambda_3\,\lambda_1}\Big[\frac{1}{2}\,\sqrt{\left(2+a+2\,b\right)\left( 2-2\,a-b\right)}-1+\frac{- e^{a+b}\left(a+2\,b\right) -\left(2\,a+b\right)}{1-e^{a+b}} \Big].\notag
 \end{align}
We therefore need to show that for all $a,b\in\mathbb{R}$ with $a^2+b^2+a\, b\leq 1$, the following inequalities, corresponding to conditions \eqref{DBE1}, \eqref{DBE3} and \eqref{DBE4} for $i=1,j=3$, hold:
 \begin{align}
 2+a-b\geq 0&\geq 0\,,\label{eq:ineqDac3d1}\\
 \frac{1}{2}\,\sqrt{\left(2+a+2\,b\right) \left( 2-2\,a-b\right) } +1+\frac{- e^{a+b}\left(a+2\,b\right) +\left(2\,a+b\right)}{1+e^{a+b}} &\geq0\,,  \label{eq:ineqDac3d2}\\
 \frac{1}{2}\, \sqrt{ \left(2+a+2\,b\right) \left( 2-2\,a-b\right) } -1+\frac{- e^{a+b}\left(a+2\,b\right) -\left(2\,a+b\right)}{1-e^{a+b}} &\geq0 \qquad  \text{if }\; a+b\neq 0\,. \label{eq:ineqDac3d3}
 \end{align}
Again, explicitly writing out the required inequalities for all $i,j\in\{1,2,3\}$ with $i\neq j$ would simply yield inequalities which can be transformed into \eqref{eq:ineqDac3d1}--\eqref{eq:ineqDac3d3} via the transformations \eqref{substitution}.
Using the further substitution
\[
a=\frac{p}{\sqrt{6}}(\sqrt{3}\cos \theta+\sin \theta)=p\, \sqrt{\frac{2}{3}}\,\cos\left(\theta-\frac{\pi}{6}\right)\,,\qquad
b=\frac{p}{\sqrt{6}}(-\sqrt{3}\cos \theta+\sin \theta)=-p\, \sqrt{\frac{2}{3}}\,\cos\left(\theta+\frac{\pi}{6}\right)\,,\notag
\]
we find
\begin{align}
2(a^2+b^2+ab)=2\,\frac{p^2}{6}(3\,\cos^2 \theta+3\,\sin^2 \theta)=p^2
\end{align}
and
\begin{align}
a+b&=\frac{2\,p}{\sqrt{6}}\sin \theta, \qquad\qquad\qquad\ \ a-b={\sqrt{2}\,p}\,\cos \theta,\notag\\
 a+2\, b&=\frac{p}{\sqrt{6}}(-\sqrt{3}\cos \theta+3\,\sin \theta)=\sqrt{2}\,p\, \sin\left(\theta-\frac{\pi}{6}\right),\\
 2\, a+b&=\frac{p}{\sqrt{6}}(\sqrt{3}\cos \theta+3\,\sin \theta)=\sqrt{2}\,p\, \sin\left(\theta+\frac{\pi}{6}\right).\notag
\end{align}
These substitutions imply that the point $(a,b)$ lies on the ellipse
$
2(a^2+b^2+a\, b)=p^2.
$
Note as well that for all $p\in \mathbb{R}$ the corresponding ellipse is invariant under the transformations \eqref{substitution} and that $(a,b)$ lies inside the full ellipse  $a^2+b^2+a\, b=1$ if and only if $p\leq \sqrt{2}$. In the following we prove that the condition $p\leq \sqrt{2}$  is sufficient for ellipticity.
\begin{figure}[t!]
\begin{minipage}[h]{0.33\linewidth}
\includegraphics[scale=0.5]{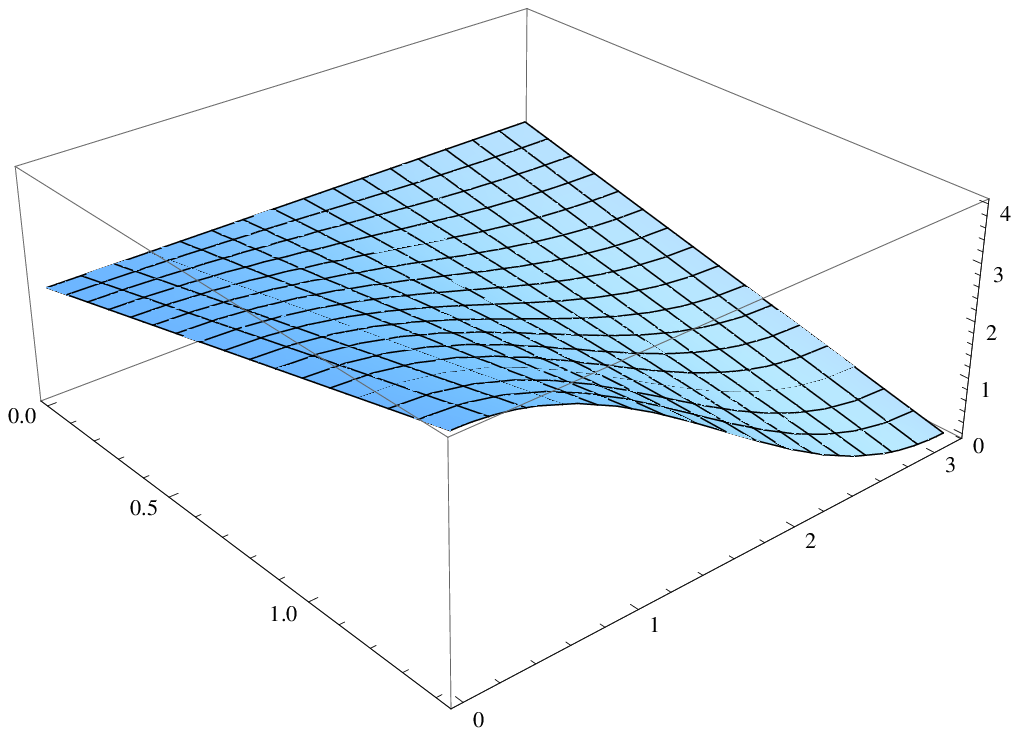}
\end{minipage}
\begin{minipage}[h]{0.33\linewidth}
\includegraphics[scale=0.55]{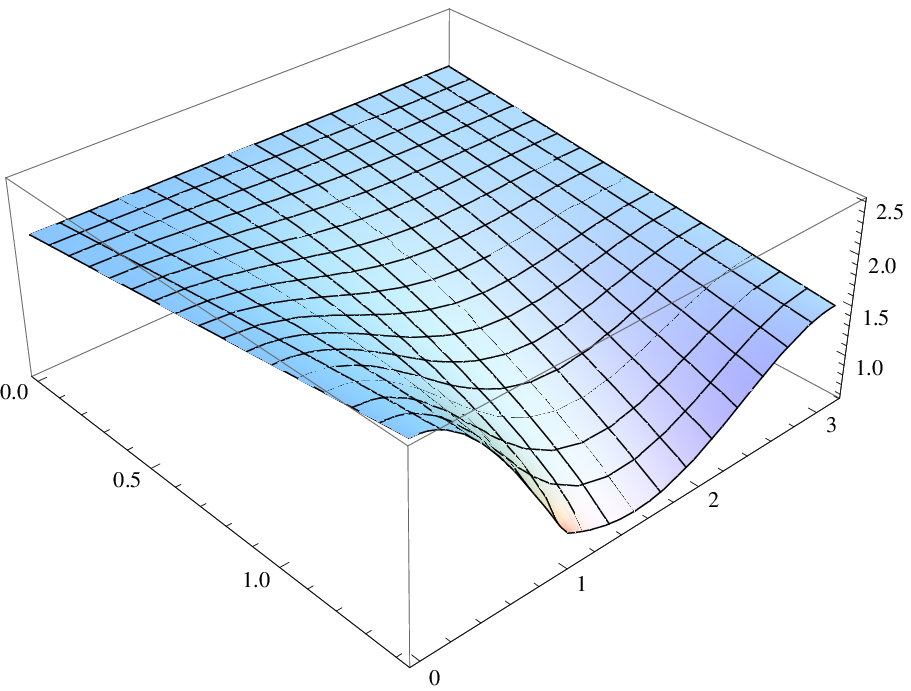}
\end{minipage}
\begin{minipage}[h]{0.33\linewidth}
\includegraphics[scale=0.55]{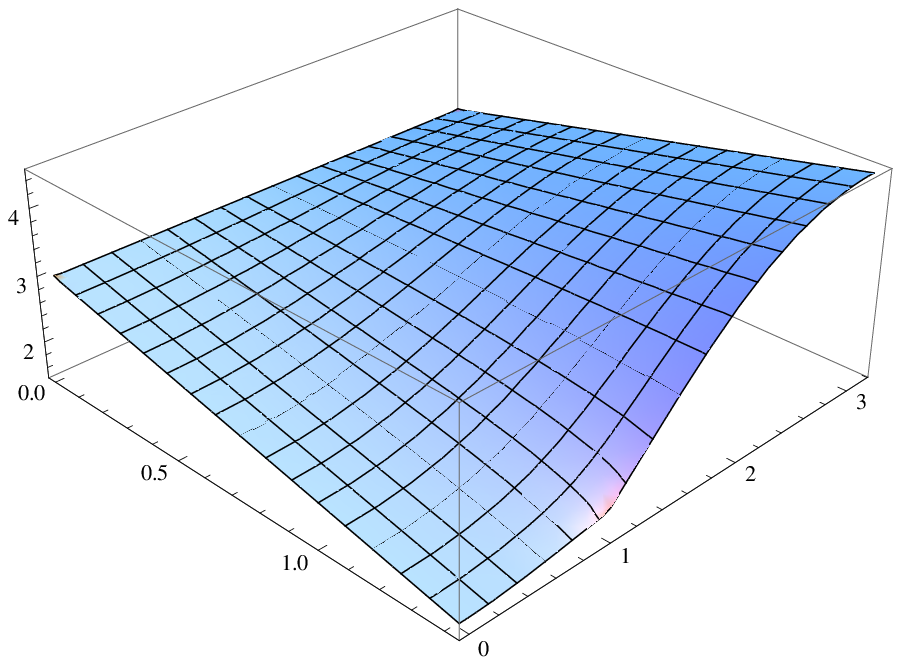}
\end{minipage}
\caption{Graphical representation of the functions $f_1,f_2, f_3$ defined in \eqref{eq:functionsDefinitions}.}
\label{graficfunctii}
\end{figure}%

In terms of $p$ and $\theta$, the required ineqalities \eqref{eq:ineqDac3d1}--\eqref{eq:ineqDac3d3} can now be written as
\begin{align}
\label{eq:functionsDefinitions}
f_1(p,\theta)\colonequals&\,2+{\sqrt{2}\,p}\,\cos \theta\geq0 \qquad\qquad\qquad\qquad\qquad\qquad\qquad\qquad\qquad\ \  \forall \,(p,\theta)\in[0,\sqrt{2}]\times [0,2\,\pi),\notag\\
f_2(p,\theta)\colonequals&\, \frac{1}{2}\,\sqrt{\left[2+\sqrt{2}\,p\, \sin\left(\theta-\frac{\pi}{6}\right)\right]\, \left[2-\sqrt{2}\,p\, \sin\left(\theta+\frac{\pi}{6}\right)\right] } \\&
+1+\frac{ -e^{\frac{2\,p}{\sqrt{6}}\sin \theta}\sqrt{2}\,p\, \sin\left(\theta-\frac{\pi}{6}\right) +\sqrt{2}\,p\, \sin\left(\theta+\frac{\pi}{6}\right)}{1+e^{\frac{2\,p}{\sqrt{6}}\sin \theta}}\geq0\qquad \forall \,(p,\theta)\in[0,\sqrt{2}]\times [0,2\,\pi)\,,\notag
\\
f_3(p,\theta)\colonequals&  \,\frac{1}{2}\,\sqrt{\left[2+\sqrt{2}\,p\, \sin\left(\theta-\frac{\pi}{6}\right)\right]\, \left[2-\sqrt{2}\,p\, \sin\left(\theta+\frac{\pi}{6}\right)\right] } \notag\\&
-1+\frac{ -e^{\frac{2\,p}{\sqrt{6}}\sin \theta}\sqrt{2}\,p\, \sin\left(\theta-\frac{\pi}{6}\right) -\sqrt{2}\,p\, \sin\left(\theta+\frac{\pi}{6}\right)}{1-e^{\frac{2\,p}{\sqrt{6}}\sin \theta}}\geq0 \qquad \forall \,(p,\theta)\in[0,\sqrt{2}]\times \{(0,2\,\pi)\setminus\{\pi\}\}\,.\notag
\end{align}
We observe that the functions $f_1, f_2, f_3$  are periodic in $\theta$ with period $2\, \pi$. Moreover, from
\begin{align}
\sin(2\, \pi-\theta)&=-\sin \theta\,, \qquad \cos(2\, \pi-\theta)=\cos \theta, \notag\\ \sin\Big(2\, \pi-\theta-\frac{\pi}{6}\Big)&=-\sin \Big(\frac{\pi}{6}+\theta\Big)\,, \qquad
\sin\Big(2\, \pi-\theta+\frac{\pi}{6}\Big)=-\sin \Big(\theta -\frac{\pi}{6}\Big)\,,
\end{align}
it follows that
\begin{align}
f_1(p,\theta)=f_1(p,2\,\pi-\theta),\qquad\qquad f_2(p,\theta)=f_2(p,2\,\pi-\theta), \qquad\qquad f_3(p,\theta)=f_3(p,2\,\pi-\theta).
\end{align}
It therefore suffices to show that the functions $f_1, f_2$ and $f_3$ are non-negative in the domains $[0,\sqrt{2}]\times[0,\pi]$, \ $[0,\sqrt{2}]\times[0,\pi]$ and $[0,\sqrt{2}]\times(0,\pi)$, respectively.

These functions are indeed non-negative on these domains, as the graphs in Fig.\ \ref{graficfunctii} clearly show; the reader can find a complete analytical proof in Appendix \ref{section:appendixAnalyticProof}.
This last assertion completes the proof of our proposition.
\end{proof}

 \section{Concluding  remarks}\setcounter{equation}{0}

In order to visualize the established domain of ellipticity for $n=3$, we go back to the initial substitution and find that the unbounded ellipticity domain given by Proposition \ref{crank3} is the set enclosed by the cone presented in Figures \ref{f3} and \ref{f4}, which is completely defined by
\begin{align}
\lambda_1=u\, e^{\frac{p}{\sqrt{6}}\left(\sqrt{3} \cos \theta+\sin \theta\right)}\,,\quad
\lambda_2=u\,,\quad
\lambda_3=u\,  e^{\frac{p}{\sqrt{6}}\left(\sqrt{3} \cos \theta-\sin \theta\right)}\,, \qquad \theta\in[0,2\pi), \;\; u\in[0,\infty), \;\; p\in[0,\sqrt{2}]\,. \notag
\end{align}
\begin{figure}[t!]
\hspace*{1cm}
\begin{minipage}[h]{0.4\linewidth}
\centering
\includegraphics[scale=0.6]{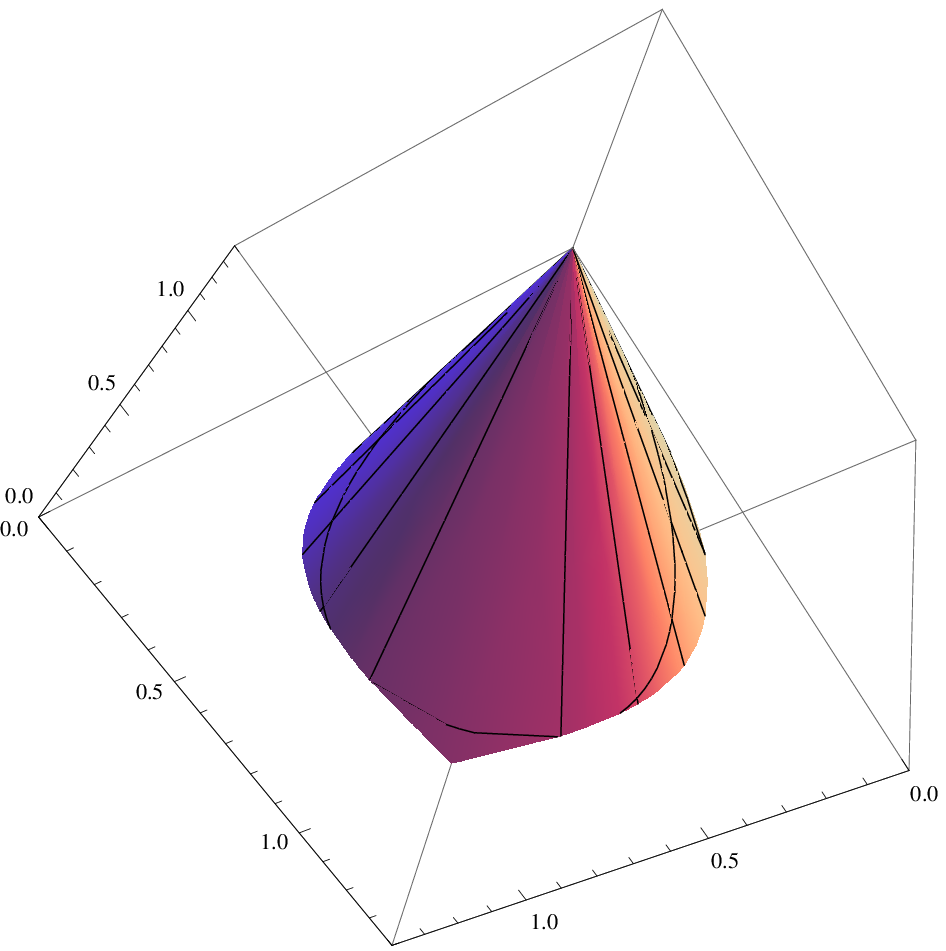}
\centering
\caption{The obtained ellipticity domain of the energy $\|\dev_3\log U\|^2$ in the principal stretches $\lambda_1,\lambda_2,\lambda_3$.}
\label{f3}
\end{minipage}
\hspace*{1cm}
\begin{minipage}[h]{0.4\linewidth}
\centering
\includegraphics[scale=0.6]{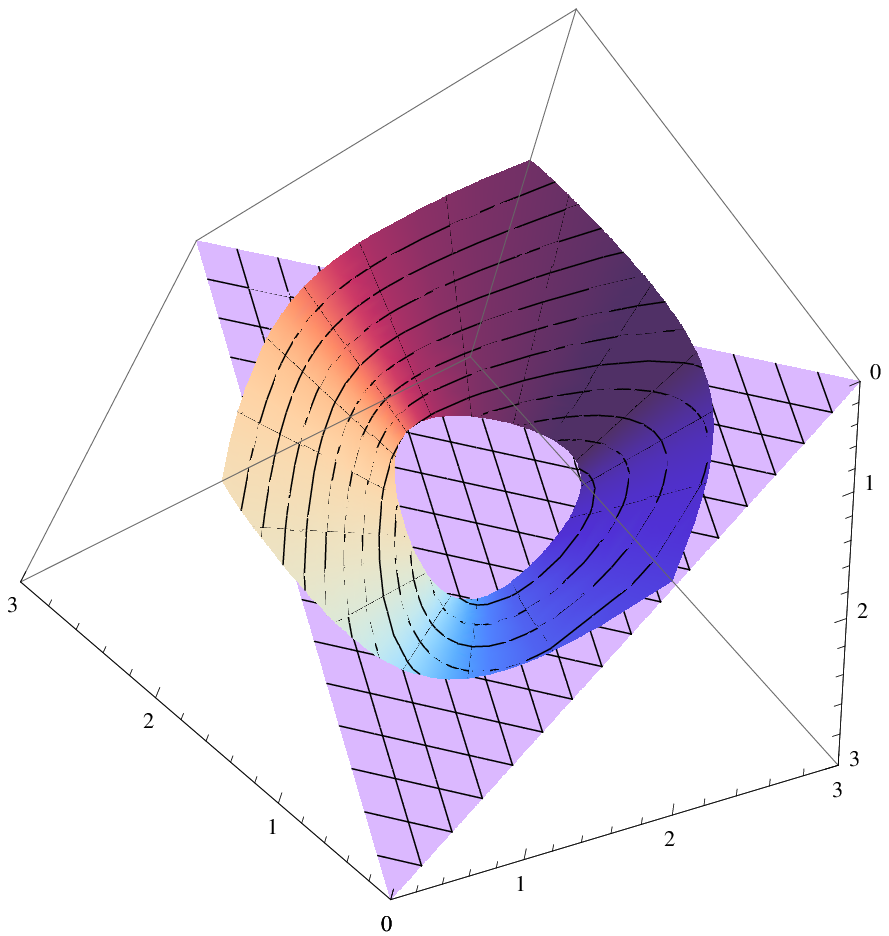}
\put(-74,73){$\bullet$}
\put(-81,73){$\id$}
\caption{The intersection of the ellipticity domain with the plane $\lambda_1+\lambda_2+\lambda_3=3$. }
\label{f4}
\end{minipage}
\end{figure}%
It is clear that the ellipticity of the energy $F\mapsto \mu\,\|\dev_3\log U\|^2$ on ${\rm GL}^+(3)$ for $\mu\geq0$ in the domain
\begin{align}
\mathcal{E}_3\left(W_{_{\rm H}}^{\rm iso}, {\rm LH}, U, \frac{2}{3}\right)=\left\{U\in{\rm PSym}(3)\, |\, \|\dev_3\log U\|^2\leq \frac{2}{3}\right\},
\end{align}
implies the ellipticity in this domain for the exponentiated Hencky energy
\begin{align}\hspace{-2mm}
 W_{_{\rm eH}}\left(F\right)\colonequals\dd{\mu}\,e^{\|{\rm dev}_3\log U\|^2}+\dd\frac{\kappa}{{\text{}}{2\, {\widehat{k}}}}\,e^{\widehat{k}\,[\tr(\log U)]^2}\,, \qquad \mu,\kappa\geq0\,,\quad \widehat{k}\geq \frac{1}{8},
\end{align}
since on the one hand $t\mapsto\mu\, e^{t}$ is convex and monotone increasing, and therefore the composition with this mapping preserves ellipticity, and on the other hand
the function
$
 F\mapsto e^{\widehat{k}\,(\log\det F)^2}
$
is rank-one convex on ${\rm GL}^+(3)$ for  $\widehat{k}\geq \frac{1}{8}$ (see \cite{NeffGhibaLankeit} for more details).

However, numerical tests suggest that the ellipticity domain of the  exponentiated Hencky energy is far bigger than all ellipticity domains which are known for various energies of  quadratic Hencky energy type, see also \cite{NeffGhibaLankeit,NeffGhibaPlasticity}. This  remark  might be useful in the study of large deformations which do not belong to the known ellipticity domains of the energies of the quadratic Hencky energy type.

  Note that the ellipticity domain $\mathcal{E}_3\left(W_{_{\rm H}}^{\rm iso}, {\rm LH}, U, \frac{2}{3}\right)$ conforms exactly to the von-Mises-Huber-Hencky criterion, also known as the \emph{maximum distortion strain energy criterion} in elasto-plasticity. Based on the results of the present paper, it is clear that the quadratic Hencky energy coupled to multiplicative plasticity will never lose LH-ellipticity in elastic unloading. This claim has been detailed in \cite{NeffGhibaPlasticity,NeffGhibaCompIUTAM,NeffGhibaAdd}. We do not know of any other elastic energy in which the ellipticity domain and the elastic domain coincide in this way.

\bibliographystyle{plain} 
\addcontentsline{toc}{section}{References}

\begin{footnotesize}

\end{footnotesize}

\appendix
\section{Appendix}
\addtocontents{toc}{\protect\setcounter{tocdepth}{0}}
\subsection{The positivity of the functions $f_1, f_2, f_3$ in Proposition \ref{crank3}}
\footnotesize
\label{section:appendixAnalyticProof}
It remains to show that the functions $f_1,f_2,f_3$ from the proof of Proposition \ref{crank3} are non-negative on the domains $[0,\sqrt{2}]\times[0,\pi]$, \ $[0,\sqrt{2}]\times[0,\pi]$ and $[0,\sqrt{2}]\times(0,\pi)$, respectively. The condition
$
 f_1(p,\theta)\geq 0
$
simply reads
\begin{align}
 \sqrt{2}\geq -p\, \cos \theta \qquad  \forall\, \theta\in[0, \pi]
\end{align}
and is obviously satisfied for all $p\in[0,\sqrt{2}]$. In order to prove that $f_3(p,\theta)\geq 0$ for all $p\in[0,\sqrt{2}]$ and $\theta\in(0,\pi)$, we will prove more: we will show that the inequality
\begin{align}\label{final1}
-1+\frac{ -e^{\frac{2\,p}{\sqrt{6}}\sin \theta}\sqrt{2}\,p\, \sin\left(\theta-\frac{\pi}{6}\right) -\sqrt{2}\,p\, \sin\left(\theta+\frac{\pi}{6}\right)}{1-e^{\frac{2\,p}{\sqrt{6}}\sin \theta}}\geq0
\end{align}
holds for all $p\in[0,\sqrt{2}]$ and $\theta\in(0,\pi)$. Note that for $(p,\theta)\in [0,\sqrt{2}]\times (0,\pi)$ we have
$e^{\frac{2\,p}{\sqrt{6}}\sin \theta}\geq 1$, since $\sin \theta\geq 0$ for $\theta\in (0,\pi)$. Therefore, instead of proving the inequality \eqref{final1}, it is enough to show that
\begin{align}\label{final2}
1-e^{\frac{2\,p}{\sqrt{6}}\sin \theta}+\sqrt{2}\,p\, e^{\frac{2\,p}{\sqrt{6}}\sin \theta}\, \sin\left(\theta-\frac{\pi}{6}\right) +\sqrt{2}\,p\, \sin\left(\theta+\frac{\pi}{6}\right)\geq0
\end{align}
for all $p\in[0,\sqrt{2}]$ and $\theta\in(0,\pi)$. To this aim, let us remark that for $\theta\in\left[\frac{\pi}{6},\pi\right)$, we have
 \begin{align}\label{final3}
1-e^{\frac{2\,p}{\sqrt{6}}\sin \theta}+&\sqrt{2}\,p\, e^{\frac{2\,p}{\sqrt{6}}\sin \theta}\, \sin\left(\theta-\frac{\pi}{6}\right) +\sqrt{2}\,p\, \sin\left(\theta+\frac{\pi}{6}\right)\notag\\
\geq &1-e^{\frac{2\,p}{\sqrt{6}}\sin \theta}+\sqrt{2}\,p\,  \sin\left(\theta-\frac{\pi}{6}\right) +\sqrt{2}\,p\, \sin\left(\theta+\frac{\pi}{6}\right)=1-e^{\frac{2\,p}{\sqrt{6}}\sin \theta}+\sqrt{6}\,p\, \sin\theta\geq0.
\end{align}
Here we have used that
\begin{align}
\sin\left(\theta+\frac{\pi}{6}\right)+\sin\left(\theta-\frac{\pi}{6}\right)=2\,\cos\frac{\pi}{6}\sin\theta=\sqrt{3}\, \sin \theta
\end{align}
and that the function $\zeta\mapsto 1-e^{\frac{2\,\zeta}{\sqrt{6}}}+\sqrt{6}\,\zeta$ is non-negative on $[0,\sqrt{2}]$. This shows that $f_3(p,\theta)\geq0$ for all $p\in[0,\sqrt{2}]$ and $\theta\in\left[\frac{\pi}{6},\pi\right)$.

Let us now consider the inequality \eqref{final1}  for all $p\in[0,\sqrt{2}]$ and $\theta\in\left(0,\frac{\pi}{6}\right)$. We note that
\begin{align}\label{banal}
\sin\left(\theta-\frac{\pi}{6}\right)&=\frac{\sqrt{3}}{2}\sin \theta -\frac{1}{2}\cos\theta,\qquad\qquad
\sin\left(\theta+\frac{\pi}{6}\right)=\frac{\sqrt{3}}{2}\sin \theta +\frac{1}{2}\cos\theta,
\end{align}
and that  $\theta\in\left(0,\frac{\pi}{6}\right)$, $\cos \theta>0$. We introduce the substitution $\zeta=p\,\sin \theta$, $\eta=p\,\cos\theta$, and our new aim is to prove that
\begin{align}\label{ppro}
r(\zeta,\eta)\colonequals\sqrt{2} \left(\frac{\sqrt{3} \zeta}{2}+\frac{\eta}{2}\right)+\sqrt{2} e^{\frac{2 \zeta}{\sqrt{6}}} \left(\frac{\sqrt{3} \zeta}{2}-\frac{\eta}{2}\right)-e^{\frac{2 \zeta}{\sqrt{6}}}+1\geq 0
\end{align}
for all $\eta\in[0,\sqrt{2}]$ and $\zeta\in [0,\sqrt{2}]$. Again, this last inequality is stronger than necessary, since we are only interested in the case $\zeta^2+\eta^2\leq 2$, i.e.\ in a subdomain of $[0,\sqrt{2}]\times[0,\sqrt{2}]$. Inequality \eqref{ppro} is satisfied since
\begin{align}
\frac{
\partial r }{\partial \zeta}(\zeta,\eta)&=e^{\sqrt{\frac{2}{3}} \zeta} \left(\zeta-\frac{\eta}{\sqrt{3}}+\frac{1}{\sqrt{6}}\right)+\sqrt{\frac{3}{2}}\\
&\geq \min\limits_{\zeta,\eta\in[0,\sqrt{2}]}\left\{e^{\sqrt{\frac{2}{3}} \zeta} \left(\zeta-\frac{\eta}{\sqrt{3}}+\frac{1}{\sqrt{6}}\right)+\sqrt{\frac{3}{2}}\right\}=\sqrt{\frac{2}{3}}>0 \qquad \forall (\zeta,\eta)\in [0,\sqrt{2}]\times [0,\sqrt{2}]\notag
\end{align}
implies that
\begin{align}
r(\zeta,\eta)\geq r(0,\eta)=0 \qquad \forall (\zeta,\eta)\in [0,\sqrt{2}]\times [0,\sqrt{2}].
\end{align}
Thus the function $f_3(p,\theta)$ is also non-negative for all $p\in[0,\sqrt{2}]$ and $\theta\in\left(0,\frac{\pi}{6}\right)$. Combining this with the earlier result for $\theta\in\left[\frac{\pi}{6},\pi\right)$, we find $f_3(p,\theta)\geq0$ for all  $(p,\theta)\in[0,\sqrt{2}]\times(0,\pi)$.

In a similar way we remark that
\begin{align}
1+\frac{ -e^{\frac{2\,p}{\sqrt{6}}\sin \theta}\sqrt{2}\,p\, \sin\left(\theta-\frac{\pi}{6}\right) +\sqrt{2}\,p\, \sin\left(\theta+\frac{\pi}{6}\right)}{1+e^{\frac{2\,p}{\sqrt{6}}\sin \theta}}\geq0\notag
\end{align}
is equivalent to
\begin{align}
1+\sqrt{2}\,p\, \sin\left(\theta+\frac{\pi}{6}\right)+e^{\frac{2\,p}{\sqrt{6}}\sin \theta}
\left[1-\sqrt{2}\,p\, \sin\left(\theta-\frac{\pi}{6}\right)\right] \geq0\notag
\end{align}for all  $\theta\in[0,\pi]$. Using that
\begin{align}
\sin\left(\theta+\frac{\pi}{6}\right)-\sin\left(\theta-\frac{\pi}{6}\right)=2\,\sin\frac{\pi}{6}\cos\theta,
\end{align}
we find
\begin{align}
\sin\left(\theta+\frac{\pi}{6}\right)\geq \sin\left(\theta-\frac{\pi}{6}\right)\qquad \forall \, \theta\in\left[0,\frac{\pi}{2}\right]\,.
\end{align}
Thus, for all $\theta\in\left[0,\frac{\pi}{2}\right]$ and all $p\in[0,\sqrt{2}]$, we have
\begin{align}
1+\sqrt{2}\,p\, \sin\left(\theta+\frac{\pi}{6}\right)&+e^{\frac{2\,p}{\sqrt{6}}\sin \theta}
\left[1-\sqrt{2}\,p\, \sin\left(\theta-\frac{\pi}{6}\right)\right] \geq 1+e^{\frac{2\,p}{\sqrt{6}}\sin \theta}+ \sqrt{2}\,p\, \sin\left(\theta-\frac{\pi}{6}\right)
\left[1-e^{\frac{2\,p}{\sqrt{6}}\sin \theta}\right]\notag\\
&= 1+e^{\frac{2\,p}{\sqrt{6}}\sin \theta}+ \frac{\sqrt{6}}{2}\,p\, \sin\theta
\left[1-e^{\frac{2\,p}{\sqrt{6}}\sin \theta}\right]-\frac{\sqrt{2}}{2}\,p\, \cos\theta
\left[1-e^{\frac{2\,p}{\sqrt{6}}\sin \theta}\right].\notag
\\
&\geq 1+e^{\frac{2\,p}{\sqrt{6}}\sin \theta}+ \frac{\sqrt{6}}{2}\,p\, \sin\theta
\left[1-e^{\frac{2\,p}{\sqrt{6}}\sin \theta}\right]\geq 0,\notag
\end{align}
since $\cos \theta\geq 0$ and $1-e^{\frac{2\,p}{\sqrt{6}}\sin \theta}\leq 0$ on $\left[0,\frac{\pi}{2}\right]$ and since  the function  $\zeta\mapsto 1+e^{\frac{2\,\zeta}{\sqrt{6}}}+ \frac{\sqrt{6}}{2}\,\zeta
\left[1-e^{\frac{2\,\zeta}{\sqrt{6}}}\right]$ is positive on $[0,\sqrt{2}]$, where the substitution $\zeta=p\, \sin \theta$ was considered. The above inequality shows that  $f_2(p,\theta)\geq 0$ for all $p\in[0,\sqrt{2}]$ and $\theta\in\left[0,\frac{\pi}{2}\right]$.

 Next we prove that $f_2(p,\theta)\geq0$ for $\theta\in \left(\frac{\pi}{2},\pi\right]$. We find
\begin{align}
\left(1+e^{\frac{2\,p}{\sqrt{6}}\sin \theta}\right)f_2(p,\theta)\geq & \frac{1}{2}\,\left(1+e^{\frac{2\,p}{\sqrt{6}}\sin \theta}\right)\sqrt{\left[2+\sqrt{2}\,p\, \sin\left(\theta-\frac{\pi}{6}\right)\right]\, \left[2-\sqrt{2}\,p\, \sin\left(\theta-\frac{\pi}{6}\right)\right] } \notag\\&
+\left(1+e^{\frac{2\,p}{\sqrt{6}}\sin \theta}\right) -e^{\frac{2\,p}{\sqrt{6}}\sin \theta}\sqrt{2}\,p\, \sin\left(\theta-\frac{\pi}{6}\right) +\sqrt{2}\,p\, \sin\left(\theta+\frac{\pi}{6}\right),\notag\\
= & \frac{1}{2}\,\left(1+e^{\frac{2\,p}{\sqrt{6}}\sin \theta}\right)\sqrt{4-{2}\,p^2\, \sin^2\left(\theta-\frac{\pi}{6}\right)}
+\left(1+e^{\frac{2\,p}{\sqrt{6}}\sin \theta}\right) \\
&\notag-e^{\frac{2\,p}{\sqrt{6}}\sin \theta}\sqrt{2}\,p\, \sin\left(\theta-\frac{\pi}{6}\right) +\sqrt{2}\,p\, \sin\left(\theta+\frac{\pi}{6}\right)\notag
\end{align}
for all  $p\in[0,\sqrt{2}]$ and $\theta\in\left(\frac{\pi}{2},\pi\right]$, since
\begin{align}
\sin\left(\theta+\frac{\pi}{6}\right)\leq \sin\left(\theta-\frac{\pi}{6}\right)\qquad \forall \, \theta\in\left(\frac{\pi}{2},\pi\right].
\end{align}
Combining \eqref{banal} with the fact that $\cos \theta<0$ for all $\theta\in\left(\frac{\pi}{2},\pi\right]$,
we introduce a new substitution $\zeta=p\,\sin \theta$,\; $\varpi=-p\,\cos\theta$,
and our new goal is to prove that
\begin{align}
h(\zeta,\varpi)\colonequals&\frac{1}{2}\,\left(1+e^{\frac{2\,\zeta}{\sqrt{6}}}\right)\sqrt{4-{2}\, \left(\frac{\sqrt{3}}{2}\zeta -\frac{1}{2}\,\varpi\right)^2}
+\left(1+e^{\frac{2\,\zeta}{\sqrt{6}}}\right) \notag\\&\dd -e^{\frac{2\,\zeta}{\sqrt{6}}}\sqrt{2}\, \left(\frac{\sqrt{3}}{2}\zeta +\frac{1}{2}\,\varpi\right) +\sqrt{2}\, \left(\frac{\sqrt{3}}{2}\zeta -\frac{1}{2}\,\varpi\right)=\frac{1}{2}\, \left(e^{\frac{2 \zeta}{\sqrt{6}}}+1\right)\,s(\zeta,\varpi) \geq 0
\end{align}
for all $\zeta,\varpi\in [0,\sqrt{2}]$, where
\begin{align}
 s(\zeta,\varpi)\colonequals  \sqrt{-\frac{3 \zeta^2}{2}+\sqrt{3} \zeta  \varpi-\frac{\varpi^2}{2}+4}-\sqrt{2} \varpi+2- \sqrt{6}\,  \zeta\, \frac{e^{\frac{2 \zeta}{\sqrt{6}}}-1}{e^{\frac{2 \,\zeta}{\sqrt{6}}}+1}\,.
\end{align}

Again, this is more than is needed, since for the non-negativity of $f_2$ for $p\in[0,\sqrt{2}]$ and $\theta\in\left(\frac{\pi}{2},\pi\right]$ it is enough to prove that $h(\zeta,\varpi)\geq 0$ only for all $\zeta,\varpi\in [0,\sqrt{2}]$ which belong to the smaller domain $\zeta^2+\varpi^2<2$. We observe that
\begin{align}
\frac{\partial s}{\partial \varpi} (\zeta,\varpi)&=\frac{1}{\sqrt{2}}\,\left(\frac{\sqrt{3} \zeta-\varpi}{\sqrt{-3 \zeta^2+2 \sqrt{3} \zeta \,\varpi-\varpi^2+8}}-2\right)
\leq \max_{\zeta,\varpi\in [0,\sqrt{2}]}\left\{\frac{1}{\sqrt{2}}\,\left(\frac{\sqrt{3} \zeta-\varpi}{\sqrt{-3 \zeta^2+2 \sqrt{3} \zeta \,\varpi-\varpi^2+8}}-2\right)\right\}\notag\\
&=\frac{\sqrt{3}-2}{\sqrt{2}}<0 \qquad \forall \, \zeta,\varpi\in [0,\sqrt{2}].
\end{align}
Hence, we deduce
\begin{align}
h(\zeta,\varpi)&\geq h(\zeta,\sqrt{2})=\sqrt{-\frac{3 \zeta^2}{2}+\sqrt{6} \zeta+3}-\frac{\sqrt{6} \left(e^{\sqrt{\frac{2}{3}} \zeta}-1\right) \zeta}{e^{\sqrt{\frac{2}{3}} \zeta}+1}\notag\\
&\geq \max_{\zeta\in [0,\sqrt{2}]}h(\zeta,\sqrt{2})=\sqrt[4]{3} \left(\sqrt{2}-2 \sqrt[4]{3} \tanh \left(\frac{1}{\sqrt{3}}\right)\right)\approx 0.0573242>0 \qquad \forall \, \zeta,\varpi\in [0,\sqrt{2}],
\end{align}
which means that $h$ is non-negative on $[0,\sqrt{2}]\times [0,\sqrt{2}]$. This last conclusion shows that $f_2(p,\theta)\geq0$ for all $(p,\theta)\in[0,\sqrt{2}]\times \left(\frac{\pi}{2},\pi\right]$. Therefore, the function  $f_2$ is non-negative on $[0,\sqrt{2}]\times \left[0,\pi\right]$.

\end{document}